\documentclass[12pt]{article}
\usepackage[dvips]{graphics}
\usepackage{bm,enumerate,amscd,amssymb,amsthm,amsmath,subfigure,graphicx,youngtab,young}
\usepackage{dsfont,pgfkeys,pgfopts}
\usepackage{pifont}
\usepackage{amsfonts}
\usepackage{mathrsfs}
\usepackage{xypic}
\usepackage{color}

\def\Sum{\textstyle\sum\limits}
\def\ssum{\textstyle\sum}

\def\Prod{\textstyle\prod\limits}
\def\pprod{\textstyle\prod}

\def\ooplus{\textstyle\bigoplus}

\def\'{\prime}  \def\''{\prime\prime}

\def\det{\mbox{\rm det}}
\def\And{\mbox{\rm ~and~}}

\def\For{\mbox{\rm ~for~}}

\def\mod{\mbox{\rm mod~}}
\def\sign{\mbox{\rm sign}}

\def\({\left(}\def\){\right)}

\def\row{\mbox{\rm row}}
\def\col{\mbox{\rm col}}

\def\dd{{\scriptscriptstyle[d]}}
\def\d'{{\scriptscriptstyle[d']}}

\theoremstyle{plain}
\newtheorem{theorem}{Theorem}[section]
\newtheorem{lemma}[theorem]{Lemma}
\newtheorem{proposition}[theorem]{Proposition}

\newtheorem{corollary}[theorem]{Corollary}
\newtheorem{example}[theorem]{Example}

\newtheorem{remark}[theorem]{Remark}

\numberwithin{equation}{section}
\begin{document}
\title{\bf{Truncated Homogeneous Symmetric Functions}}

\author{Houshan Fu\\
\small School of Mathematics and Econometrics\\
\small Hunan University\\
\small Changsha, Hunan, China\\
\small fuhoushan@hnu.edu.cn \\ \\
\and
Zhousheng Mei\thanks{Supported by Hunan Provincial Innovation Foundation for Postgraduate (CX2018B215)}\\
%\thanks{Supported by Hunan Provincial Innovation Foundation for Postgraduate (CX2018B215)}\\
\small College of Mathematics and Physics\\
\small Wenzhou University \\
\small Wenzhou, Zhejiang, China\\
\small 20190316@wzu.edu.cn\\ \\
%\and
%Suijie Wang\thanks{Supported by NSFC 11871204}\\
%\small School of Mathematics and Econometrics \\
%\small Hunan University \\
%\small Changsha, Hunan, China\\
%\small wangsuijie@hnu.edu.cn
%\and
%Yue Zhou\\
%\small School of Mathematics and Statistics\\
%\small Central South University\\
%\small Changsha, Hunan, China\\
%\small zhouyue@csu.edu.cn
}

\date{}
\maketitle
\begin{abstract}
Extending the elementary and complete homogeneous symmetric functions, we introduce the truncated homogeneous symmetric function $h_{\lambda}^{\dd}$ in  $(\ref{THSF})$ for any integer partition $\lambda$, and show that the transition matrix from $h_{\lambda}^{\dd}$ to the power sum symmetric functions $p_\lambda$ is given by \[M(h^{\dd},p)=M'(p,m)z^{-1}D^{\dd},\]
where $D^{\dd}$ and $z$ are nonsingular diagonal matrices. Consequently,  $\{h_{\lambda}^{\dd}\}$ forms a basis of the ring $\Lambda$ of symmetric functions. In addition, we show that the generating function $H^{\dd}(t)=\ssum_{n\ge 0}h_n^{\dd}(x)t^n$  satisfies
\[\omega(H^{\dd}(t))=\left(H^{\dd}(-t)\right)^{-1},\]
where $\omega$ is the involution of $\Lambda$ sending each elementary symmetric function $e_\lambda$ to the complete homogeneous symmetric function $h_\lambda$.

\end{abstract}

\section{Introduction}

Let $\Lambda=\ooplus_{n\ge 0}\Lambda^n$ (or $\Lambda[x]=\ooplus_{n\ge 0}\Lambda^n[x]$) be the ring of symmetric functions in $x$ over $\Bbb{Q}$ (or $\Bbb{Z}$, $\Bbb{R}$,  $\Bbb{C}$) consisting of those formal power series of $x$ who are invariant under permutations of the independent variable sequence $x=(x_{1},x_{2},\ldots)$, and $\Lambda^n$ the subspace of homogeneous symmetric functions of degree $n$. Two classical bases of $\Lambda$ are elementary and complete homogeneous symmetric functions,  denoted $e_\lambda(x)$ and $h_\lambda(x)$ respectively and  defined as follows, for any integer partition $\lambda=(\lambda_1\ge \lambda_2\ge \cdots\ge \lambda_l\ge 0)$,
\begin{eqnarray*}
  &&E(t)=\Sum_{n\ge 0}e_n(x)t^n=\Prod_{k\ge 1}(1+x_kt)\quad \And \quad e_\lambda(x)=\Prod_i e_{\lambda_i}(x),\\
  &&H(t)=\Sum_{n\ge 0}h_n(x)t^n=\Prod_{k\ge 1}(1+x_kt+x_k^2t^2+\cdots )\quad \And \quad h_\lambda(x)=\Prod_i h_{\lambda_i}(x).
\end{eqnarray*}
Inspired by the above definitions, for any positive integer $d$, we introduce the {\bf truncated homogeneous symmetric functions} $h_{\lambda}^{{\dd}}$ defined by
\begin{equation}\label{THSF}
H^{\dd}(t)=\Sum_{n\ge 0}h_n^{\dd}(x)t^n=\Prod_{k\ge 1}(1+x_kt+x_k^2t^2+\cdots+x_k^dt^d)\quad \And \quad h_\lambda^{\dd}(x)=\Prod_i h_{\lambda_i}^{\dd}(x),
\end{equation}
which extends elementary and complete homogeneous symmetric functions by $e_{\lambda}(x)=h_{\lambda}^{\scriptscriptstyle[1]}(x)$ and $h_{\lambda}(x)=h_{\lambda}^{\scriptscriptstyle[\infty]}(x)$.

In the exercise 7.91 of \cite{Stanley}, Stanley proposed the formal power series
\[
F(t)=\Sum_{n\geq 0}f_{n}(x)t^{n} \quad{\rm with}\quad f_{0}=1,
\]
and defined the Schur function $s_{\lambda}^F(x)$ as the coefficient of $s_{\lambda}(t_1,t_2,\ldots)$  in the expansion of $F(t_1)F(t_2)\cdots$. Here we clarify that in general, $f_n(x)$ is  not the forgotten symmetric functions of Macdonald \cite{Macdonald}.  The exercise 7.91 concludes that $s_{\lambda}^F(x)$ extends the canonical Schur function $s_\lambda(x)$, and presents some other basic properties on $s_{\lambda}^F(x)$ by specializing $F(t)$.  More results on the formal power series $F(t)$ and  the Schur function $s_{\lambda}^F(x)$ can be found in \cite{King,Ram1991,Ram1996}.

The truncated homogeneous symmetric functions $h^\dd_n$ is nothing but defined by the specialization $F(t)=H^{\dd}(t)$. In this paper, we make a systematic study on $h^{\dd}_{\lambda}$. We show in Section 2 that for each positive integer $d$, the truncated homogeneous  functions $h^{\dd}_{\lambda}$ form a basis of the ring $\Lambda$. Also, we present the transition matrices from the basis $\{h^{\dd}_{\lambda}:\lambda\vdash n\}$ to the classical bases of $\Lambda^n$, including the monomial symmetric functions $\{m_\lambda:\lambda\vdash n\}$, the complete homogeneous symmetric functions $\{h_{\lambda}:\lambda\vdash n\}$, the elementary symmetric functions $\{e_{\lambda}:\lambda\vdash n\}$,  the power sum symmetric functions $\{p_{\lambda}:\lambda\vdash n\}$, and  the Schur functions $\{s_{\lambda}:\lambda\vdash n\}$, see \cite{Macdonald,Stanley}  more details on these symmetric functions. In Section 3, we first show that $\omega(H^{\dd}(t))=
\left(H^{\dd}(-t)\right)^{-1}$ if $\omega$ is the involution of $\Lambda$ sending each elementary symmetric function $e_\lambda$ to the complete homogeneous symmetric function $h_\lambda$. Next we introduce the endomorphism $\omega^\dd$ of $\Lambda$ with $\omega^\dd(e_n)=h_n^\dd$ whose eigenvectors turn out to be all power sum symmetric functions, and obtain $\omega^\dd\omega^\d'=\omega^\d'\omega^\dd$ for any positive integers $d$ and $d'$.

Most recently, Grinberg \cite{Grinberg} independently introduced the function $h_n^{\dd}(x)$,  which was denoted by $G(d+1,n)$ and called the $(d+1,n)$-Petrie symmetric functions. Some results of this paper have been presented in \cite{Grinberg}, but obtained in a different method.   More precisely, our Proposition  \ref{pro-coefficient} recovers Theorem 4.4 and answers Questions 4.7 of \cite{Grinberg}; our Theorem \ref{main-1} is equivalent to Theorem 6.3 of \cite{Grinberg};   the proof of Theorem \ref{main-2} implies Theorem 6.2 of \cite{Grinberg}. Please see \cite{Grinberg1} for more detailed results on Petrie symmetric functions.

\section{Truncated Homogeneous Symmetric Functions}

From the definition $(\ref{THSF})$, it is clear that  the $n$-th truncated homogeneous symmetric function  $h^{\dd}_{n}$ can be written as
\begin{eqnarray}\label{def1}
h^{\dd}_{n}(x)=\Sum_{0\leq\alpha_{i}\leq d\atop \sum\alpha_{i}=n}x^{\alpha_{1}}_{1}x^{\alpha_{2}}_{2}\ldots\;.
\end{eqnarray}
It follows that $h^{\scriptscriptstyle[1]}_{n}=e_{n}$ and $h^{\scriptscriptstyle[\infty]}_{n}=h_{n}$. Then $h^{\scriptscriptstyle[1]}_{\lambda}=e_{\lambda}$ and $h^{\scriptscriptstyle[\infty]}_{\lambda}=h_{\lambda}$ consequently.

Let $A=(a_{ij})_{i,j\geq 1}$ be an integer matrix with finitely many nonzero entries and with row and column sums
\begin{eqnarray*}
r_{i}=\Sum_{j}a_{ij}\quad\quad\And\quad\quad c_{i}=\Sum_{i}a_{ij}.
\end{eqnarray*}
Define the row-sum vector $\row(A)$ and column-sum vector $\col(A)$ to be
\begin{eqnarray*}
\row(A)=(r_{1},r_{2},\ldots) \quad\quad \And\quad\quad \col(A)=(c_{1},c_{2},\ldots).
\end{eqnarray*}
Given a positive integer $d$, we say $A=(a_{ij})_{i,j\geq 1}$ is a $[0,d]$-matrix if each $ a_{ij}$ is an integer with  $0\leq a_{ij}\leq d$. Denote by $M^{\dd}_{\lambda\mu}$ the number of $[0,d]$-matrices $A$ with $\row(A)=\lambda$ and $\col(A)=\mu$. From (\ref{def1}), we immediately have

\begin{proposition}\label{m-expansion}
If $d$ is a positive integer and $\lambda\vdash n$, we have
\begin{eqnarray}\label{m-eqn}
h^{\dd}_{\lambda}=\Sum_{\mu\vdash n}M_{\lambda\mu}^{\dd}m_{\mu}.
\end{eqnarray}
In particular,
$
e_{\lambda}=\Sum_{\mu\vdash n}M^{\scriptscriptstyle[1]}_{\lambda\mu}m_{\mu}  \And h_{\lambda}=\Sum_{\mu\vdash n}M^{\scriptscriptstyle[\infty]}_{\lambda\mu}m_{\mu}.
$
\end{proposition}
\begin{proof}
Let $\mu$ be a partition of $n$. Consider the coefficient of a monomial $x^{\mu}$ in $h^{\dd}_{\lambda}=h^{\dd}_{\lambda_{1}}
h^{\dd}_{\lambda_{2}}\ldots$\;. Each monomial in $h^{\dd}_{\lambda_{i}}$ is of the form $\Prod_{j}x_{j}^{a_{ij}}$, where  $0\leq a_{ij}\leq d$ and $\Sum_{j}a_{ij}=\lambda_{i}$. Hence we must have
\[
\Prod_{i,j}x_{j}^{a_{ij}}=\Prod_{j}x_{j}^{\mu_{j}}
\]
so that $\Sum_{i}a_{ij}=\mu_{j}$. Since $h^{\dd}_{\lambda}$ is a symmetric function, the result holds.
\end{proof}
\begin{corollary}
With above notations, we have $M_{\lambda\mu}^{\dd}=M_{\mu\lambda}^{\dd}$.
\end{corollary}
\begin{example}
If $\lambda\vdash 4$ and $d=2$, then we have
\begin{eqnarray*}
h^{\scriptscriptstyle[2]}_{1111}&=&m_{4}+4m_{31}+6m_{22}+12m_{211}+24m_{1111}\\
h^{\scriptscriptstyle[2]}_{211}~&=&m_{4}+3m_{31}+4m_{22}+~7m_{211}+12m_{1111}\\
h^{\scriptscriptstyle[2]}_{22}~~&=&m_{4}+2m_{31}+3m_{22}+~4m_{211}+~6m_{1111}\\
h^{\scriptscriptstyle[2]}_{31}~~&=&\quad\quad ~~~m_{31}+2m_{22}+~3m_{211}+~4m_{1111}\\
h^{\scriptscriptstyle[2]}_{4}~~&=&\quad\quad\quad\quad\quad\quad ~m_{22}+~~m_{211}+~~~m_{1111}.
\end{eqnarray*}
\end{example}
Recall that the Kostka number $K_{\lambda\mu}$ is  the number of semistandard Young tableaux of shape $\lambda$ and type $\mu$ and
$h^{\scriptscriptstyle[\infty]}_{\mu}=h_{\mu}=\Sum_{\lambda}K_{\lambda\mu}s_{\lambda}  \;\text{and} \;h^{\scriptscriptstyle[1]}_{\mu}=e_{\mu}=\Sum_{\lambda}K_{\lambda'\mu}s_{\lambda}.$
\begin{remark} $h^{\scriptscriptstyle[\infty]}_{\lambda}$ and $h^{\scriptscriptstyle[1]}_{\lambda}$ are Schur-nonnegative. However, if $d\ne 0$ and $\infty$, $h^{\dd}_{\lambda}$ may not be Schur-nonnegative in general. {\rm (}e.g., $h^{\scriptscriptstyle[2]}_{3}(x)=\Sum_{i<j<k}x_{i}x_{j}x_{k}+\Sum_{i\neq j}x^{2}_{i}x_{j}=s_{21}-s_{111}.${\rm )}
\end{remark}

\begin{proposition}\label{m-expansion-y}
Given a positive integer $d$,  we have
  \begin{eqnarray*}
\Prod_{i,j}(1+x_{i}y_{j}+\cdots+x^{d}_{i}y^{d}_{j})
=\Sum_{\lambda}h^{\dd}_{\lambda}(y)m_{\lambda}(x)
=\Sum_{\lambda}h^{\dd}_{\lambda}(x)m_{\lambda}(y),
\end{eqnarray*}
where $\lambda$ and $\mu$ range over all partitions.
\end{proposition}
\begin{proof}
Recall that $M^{\dd}_{\alpha\beta}$ is the number of matrices $A=(a_{ij})$ with $0\leq a_{ij}\leq d$ and satisfying row$(A)=\alpha$ and col$(A)=\beta$.
Note that
\[
\Prod_{i,j}(1+x_{i}y_{j}+\cdots+x^{d}_{i}y_{j}^{d})=
\Sum_{\alpha,\beta}M^{\dd}_{\alpha\beta}x^{\alpha}y^{\beta}.
\]
Since $\Prod_{i,j}(1+x_{i}y_{j}+\cdots+x^{d}_{i}y_{j}^{d})$ is a symmetric function on variables $x$ and $y$, we have
\begin{eqnarray*}
\Prod_{i,j}(1+x_{i}y_{j}+\cdots+x^{d}_{i}y_{j}^{d})&=& \Sum_{\alpha,\beta}M^{\dd}_{\alpha\beta}x^{\alpha}y^{\beta}
=\Sum_{\lambda,\mu}M^{\dd}_{\lambda\mu}m_{\lambda}(x)m_{\mu}(y).
\end{eqnarray*}
From Proposition \ref{m-expansion}, the first equality holds. The second equality is obvious by exchanging $x$ and $y$.
\end{proof}
Suppose $x=\{x_{1},x_{2},\ldots\}$ and $y=\{y_{1},y_{2},\ldots\}$. Denote $xy=\{x_{i}y_{j}:x_{i}\in x,y_{j}\in y\}$. It is clear that  $p_{n}(xy)=p_{n}(x)p_{n}(y)$ from the definition of the power sum symmetric functions $p_{n}=\Sum_{i}x_i^n$. Recall from \cite{Macdonald,Stanley} the expansion
\[
\Prod_{i,j}(1+x_{i}y_{j})=
\Sum_{\lambda}\varepsilon_{\lambda}z^{-1}_{\lambda}p_{\lambda}(x)p_{\lambda}(y),
\]
where $\lambda=(1^{n_1},2^{n_2},\ldots)$ ranges over all partitions,
\[
\varepsilon_\lambda=(-1)^{|\lambda|-\ell(\lambda)} \quad\And \quad z_\lambda=\pprod_{i\geq 1}i^{n_{i}}\cdot n_{i}!.
 \]
 Applying $p_{\lambda}(xy)=p_{\lambda}(x)p_{\lambda}(y)$, we can extend the above expansion as follows,  which is a key lemma in proving Theorem \ref{main-1}.

\begin{lemma}\label{key-lemma}Suppose $x=\{x_{1},x_{2},\ldots\}$, $y=\{y_{1},y_{2},\ldots\}$, and $z=\{z_{1},z_{2},\ldots\}$ are variables. We have
\begin{eqnarray}\label{main-eqn}
\Prod_{i,j,k}(1+x_{i}y_{j}z_{k})
=\Sum_{\lambda}\varepsilon_{\lambda}z^{-1}_{\lambda}p_{\lambda}(x)p_{\lambda}(y)p_{\lambda}(z),
\end{eqnarray}
where $\lambda$ runs through all partitions.
\end{lemma}

\begin{theorem}\label{main-1}
Suppose $\lambda\vdash n$ and $p_{\lambda}=\ssum_{\mu\vdash n}R_{\lambda\mu}m_{\mu}$. We have
\[
h^{\dd}_{\lambda}=\Sum_{\mu\vdash n}z^{-1}_{\mu}D^{\dd}_{\mu}R_{\mu\lambda}p_{\mu},
\]
where
\[
D^{\dd}_{\mu}=(-d)^{\sum_{k}n_{k(d+1)}}\quad \For\quad \mu=(1^{n_{1}},2^{n_{2}},\ldots).
\]
In particular, $D^{\scriptscriptstyle[1]}_{\lambda}=\varepsilon_{\lambda}$ and $D^{\scriptscriptstyle[\infty]}_{\lambda}=1$.
Moreover, the truncated homogeneous symmetric functions $h^{\dd}_{\lambda}$ form a basis of $\Lambda[x]$.
\end{theorem}
\begin{proof}
Let $\xi=e^{2\pi i/d+1}$ be the $(d+1)$-th primitive unit root. From Lemma \ref{key-lemma}, we have
\begin{eqnarray*}
\Prod_{i,j}(1+x_{i}y_{j}+\cdots+x^{d}_{i}y^{d}_{j})
&=&\Prod_{i}\Prod_{j}\Prod_{k=1}^{d}(1-x_{i}y_{j}\xi^{k})\\
&=&\Sum_{\mu}\varepsilon_{\mu}z^{-1}_{\mu}p_{\mu}(-\xi,\ldots,-\xi^{d})p_{\mu}(x)p_{\mu}(y).
\end{eqnarray*}
Applying $p_{\mu}(y)=\ssum_{\lambda}R_{\mu\lambda}m_{\lambda}(y)$ and Proposition \ref{m-expansion-y}, we have
\begin{eqnarray*}
\Sum_{\lambda}h^{\dd}_{\lambda}(x)m_{\lambda}(y)=\Sum_{\lambda}
\bigg(\Sum_{\mu}(-1)^{\ell(\mu)}z^{-1}_{\mu}p_{\mu}(\xi,\ldots,\xi^{d})R_{\mu\lambda}p_{\mu}(x)\bigg)
m_{\lambda}(y).
\end{eqnarray*}
Since the monomial symmetric functions $m_{\lambda}(y)$ are linearly independent over $\Lambda[x]$, we have
\[
h^{\dd}_{\lambda}(x)=
\Sum_{\mu}(-1)^{\ell(\mu)}z^{-1}_{\mu}p_{\mu}(\xi,\ldots,\xi^{d})R_{\mu\lambda}p_{\mu}(x).
\]
Note that for any positive integer $n$, $p_{n}(\xi,\ldots,\xi^{d})=d$ if $d+1\mid n$, and $-1$ otherwise. It follows that $(-1)^{\ell(\mu)}p_{\mu}(\xi,\ldots,\xi^{d})=(-d)^{\sum_{k}n_{k(d+1)}}$ for $\mu=(1^{n_{1}},2^{n_{2}},\ldots)$, which implies
\[
h^{\dd}_{\lambda}=\Sum_{\mu\vdash n}z^{-1}_{\mu}D^{\dd}_{\mu}R_{\mu\lambda}p_{\mu}.
\]
Namely, the transition matrix $M(h^{\dd},p)$ from  $\{h^{\dd}_\lambda:\lambda\vdash n\}$ to  $\{p_\lambda:\lambda\vdash n\}$ is
\[
M(h^{\dd},p)=M'(p, m)z^{-1}D^{\dd},
\]
where $z^{-1}$ and $D^{\dd}$ denote the diagonal matrices whose diagonal entries are  $z^{-1}_{\mu}$ and $D^{\dd}_{\mu}$ for $\mu\vdash n$ respectively, and $M'(p, m)$ is the transpose of the transition matrix from $\{p_\lambda: \lambda\vdash n\}$ to $\{m_\lambda: \lambda\vdash n\}$. It is clear that the matrix $M(h^{\dd},p)$ is nonsingular.
So $\{h^{\dd}_\lambda: \lambda\vdash n\}$ forms a basis of $\Lambda^n$.
\end{proof}

Denote the transition matrices $R=M(p,m)$ and $K=M(s,m)$. From Chapter I of \cite{Macdonald}, we know that
\begin{eqnarray*}
M(p,h)=zR^{*},
  \quad M(p,e)=\varepsilon zR^{*},
  %M(p,f)=\varepsilon R;
 \quad \And\quad M(p,s)=RK^{-1},
\end{eqnarray*}
where $R^{*}=(R')^{-1}$ is the transposed inverse of $R$, and $\varepsilon$ (resp. z) denotes the diagonal matrix with diagonal entries $\varepsilon_{\lambda}$ (resp. $z_{\lambda}$).
From Theorem \ref{main-1}, we immediately have
\begin{corollary} \label{Bases-coro}
With above notations, we have
%\begin{enumerate}[(1).]
\begin{eqnarray*}
&&M(h^{\dd},p)=R'z^{-1}D^{\dd};\\
&&M(h^{\dd},m)=R'z^{-1}D^{\dd}R;\\
&&M(h^{\dd},h)=R'D^{\dd}R^{*};\\
&&M(h^{\dd},e)=R'D^{\dd}\varepsilon R^{*};\\
&&M(h^{\dd},s)=R'z^{-1}D^{\dd}RK^{-1}.
\end{eqnarray*}

\end{corollary}

From the exercise 7.91 of \cite{Stanley}, if $F(t)=\Sum_{n\ge 0}f_nt^n=\Sum_{i=0}^d t^i$, then
\[
\Sum_{n\ge 0}h_n^\dd(x)=\Prod_{i}F(x_i)=\Sum_\lambda s_\lambda^Fs_\lambda(x),
\]
where
\[
s_\lambda^F=\det\big(f_{\lambda_i-i+j}\big)_{1\le i,j\le \ell(\lambda)}.
\]
On the other hand, since $\Prod_{i,j}(1+x_iy_j)=\Sum_\lambda s_\lambda(y) s_{\lambda'}(x)$, we have
\begin{eqnarray*}
\Sum_{n\geq 0}h^{\dd}_{n}(x)=\Prod_{i\geq 1}\Prod_{j=1}^{d}(1-x_{i}\xi^{j})=\Sum_{\lambda }(-1)^{|\lambda|}s_{\lambda}(\xi,\ldots,\xi^d)s_{\lambda'}(x),
\end{eqnarray*}
where $\xi=e^{2\pi i/d+1}$ is the $(d+1)$-th primitive unit root. From the algebraic definition of Schur functions, we have
\[
s_{\lambda}(\xi,\ldots,\xi^d)=\xi^{|\lambda|}
\frac{\det\left(\left(\xi^{\lambda_{j}+d-j}\right)^{i-1}\right)_{1\leq i,j\leq d}}{\det\left(\left(\xi^{d-j}\right)^{i-1}\right)_{1\leq i,j\leq d}}=\xi^{|\lambda|}\prod_{1\leq i<j\leq d}\frac{\xi^{\lambda_{j}+d-j}-\xi^{\lambda_{i}+d-i}}{\xi^{d-j}-\xi^{d-i}}.
\]
Note that $s_{\lambda}(\xi,\ldots,\xi^d)\ne 0$ iff all $\lambda_{i}+d-i$ for $1\le i\le d$ are distinct in  $\Bbb{Z}/(d+1)\Bbb{Z}$. After routine calculations, we obtain that $s_{\lambda}(\xi,\ldots,\xi^d)=(-1)^{d+r}\sign(\sigma)$ if there are some permutation $\sigma\in S_d$ and $0\le r\le d$ such that $\lambda+\delta_d=(\lambda_1+d-1,\lambda_2+d-2,\ldots)$  and $\delta_{d+1}\setminus r=(d,d-1,\ldots,r+1,r-1,\ldots,1,0)$ in $\Bbb{Z}/(d+1)\Bbb{Z}$  are the same under the permutation $\sigma$, and $0$ otherwise.

\begin{proposition}\label{pro-coefficient}If $F(t)=\Sum_{n\ge 0}f_nt^n=\Sum_{i=0}^d t^i$, we have $s_\lambda^F=\det\big(f_{\lambda_i-i+j}\big)_{1\le i,j\le \ell(\lambda)}=\pm 1$ or $0$. More precisely,
\begin{eqnarray*}s_\lambda^F=
\left\{
  \begin{array}{ll}
    (-1)^{d+r}\sign(\sigma), & \hbox{if $\lambda'+\delta_d=\sigma(\delta_{d+1}\setminus r)$ in $\Bbb{Z}/(d+1)\Bbb{Z}$} {\text{\rm ~for some~}} \sigma\in S_d; \\
    0, & \hbox{otherwise.}
  \end{array}
\right.
\end{eqnarray*}
\end{proposition}

\section{The Endomorphisms $\omega^\dd$ and $\omega$}

Recall the involution $\omega:\Lambda\to \Lambda$ sending each $e_n$ to $h_n$, and
\[
H^{\dd}(t)=\Sum_{n\ge 0}h_n^{\dd}(x)t^n=\Prod_{k\ge 1}(1+x_kt+x_k^2t^2+\cdots+x_k^dt^d ).
\]
Denote $E(t)=H^{\scriptscriptstyle{[1]}}(t)$ and $H(t)=E^{-1}(-t)= H^{\scriptscriptstyle{[\infty]}}(t)$.
\begin{theorem}\label{main-2}
For any positive integer $d$, we have
\[
\omega(H^{\dd}(t))=(H^{\dd}(-t))^{-1}.
\]
In particular, $\omega(E(t))=H(t)$.
\end{theorem}
\begin{proof}
Recall from Section 8 of Chapter I in \cite{Macdonald} that for $f\in \Lambda^m$ and $g \in \Lambda^n$, the plethysm $f\circ g \in \Lambda^{m+n}$ is defined by $f\circ g=f(y_1,y_2,\ldots)$ with the variables $y_i$ given by $\Prod_i(1+y_it)=\Prod_{\alpha}(1+x^\alpha t)^{u_\alpha}$  when $g=\Sum_{\alpha}u_\alpha x^\alpha$. We have the facts
\[
f\circ p_n =f(x_1^n,x_2^n,\ldots)\quad \And \quad\omega(e_m\circ p_n)=(-1)^{m(n-1)}h_m\circ p_n.
\]
Note
\begin{eqnarray*}
H^{\dd}(t)&=&\Prod_{i}\frac{1-(x_{i}t)^{d+1}}{1-x_{i}t}
=\left(\Sum_{n\ge 0}e_{n}(-x_1^{d+1},-x_2^{d+1},\ldots)t^{n(d+1)}\right)
\left(\Sum_{n\ge 0}h_{n}t^n\right)\\
&=&\left(\Sum_{n\ge 0}(-1)^n\big( e_{n}\circ p_{d+1}\big)t^{n(d+1)}\right)
\left(\Sum_{n\ge 0}h_{n}t^n\right).
\end{eqnarray*}
It follows that
\begin{eqnarray*}
\omega \left(H^{\dd}(t)\right)&=&\left(\Sum_{n\ge 0}\big(h_{n}\circ p_{d+1}\big)(-t)^{n(d+1)})\right)
\left(\Sum_{n\ge 0}e_{n}t^n\right)\\
&=&\Prod_{i}\frac{1+x_{i}t}{1-x_{i}^{d+1}(-t)^{d+1}}
=\left(H^{\dd}(-t)\right)^{-1}.
\end{eqnarray*}
\end{proof}
Similarly, we define an isomorphism $\omega^{\dd}:\Lambda\to \Lambda$  algebraically extended by
\[
\omega^{\dd}:\Lambda\to\Lambda,\quad e_n\mapsto h_n^\dd.
\]
\begin{proposition}
For any partition  $\lambda\vdash n$ and positive integer $d$, we have
\[
\omega^{\dd}(p_{\lambda})=\varepsilon_{\lambda} D^{\dd}_{\lambda} p_{\lambda}.
\]
Namely,  each power sum symmetric function $p_{\lambda}$ is an eigenvector  of $\omega^{\dd}$ corresponding to the eigenvalue $\varepsilon_{\lambda} D^{\dd}_{\lambda}$ .
\end{proposition}
\begin{proof}Note $M(p,e)=\varepsilon z R^{*}$ and $M(h^{\dd},p)=R'z^{-1}D^{\dd}$
from Corollary \ref{Bases-coro}. Then
\[
M\big(\omega^\dd(p), p\big)=M\big(\varepsilon z R^*\omega^\dd(e), p\big)=\varepsilon z R^*M\big(h^\dd, p\big)=\varepsilon D^\dd.
\]
\end{proof}
\begin{corollary}\label{diagram-commute}For any integers $d, d'\ge 1$, we have
\[
\omega^{\dd}\omega^{[\scriptscriptstyle{d'}]}=\omega^{[\scriptscriptstyle{d'}]}\omega^{\dd}.
\]
In particular, $\omega^\dd\omega=\omega\omega^\dd$, namely the following diagram commutes,
\[
\xymatrix{
  \quad\quad  E(t)\quad  \ar[d]_{\omega} \ar[r]^{\omega^{\dd}}
                & \quad H^{\dd}(t)\quad\quad   \ar[d]^{\omega}  \\
  \quad\quad  H(t)\quad   \ar[r]^{\omega^{\dd}}
                & (H^{\dd}(-t))^{-1}.            }
\]
\end{corollary}

Denote by $M^\dd=(M^\dd_{\lambda\mu})$ the transition matrix from $\{h^\dd_\lambda: \lambda\vdash n\}$ to $\{m_\lambda:\lambda\vdash n\}$. Recall from Proposition \ref{m-expansion} that $M^{\dd}_{\lambda\mu}$ is the number of $[0,d]$-matrix $A=(a_{ij})_{i,j\geq 1}$ satisfying $\row(A)=\lambda$ and $\col(A)=\mu$.  From Corollary \ref{Bases-coro}, we have
\[
M^\dd=R'z^{-1}D^{\dd}R.
\]
Let $N^{\dd}_{\lambda\mu}$ denote the number of nonnegative integral matrices $A=(a_{ij})_{i,j\geq 1}$ such that row$(A)=\lambda$, col$(A)=\mu$ and $a_{ij}\equiv0$ or $1\;\mod(d+1)$, and the matrix $N^\dd=(N^{\dd}_{\lambda\mu})$.
\begin{corollary}
With above notations, if $d$ is odd, then $N^{\dd}$ is nonsingular and
\[
N^{\dd}=R'\varepsilon z^{-1}D^{\dd}R.
\]

\end{corollary}
\begin{proof}From Theorem \ref{main-1}, we have $h^{\dd}_{\lambda}=\Sum_{\mu}z^{-1}_{\mu}D^{\dd}_{\mu}R_{\mu\lambda}p_{\mu}$. Since $\omega(p_{\lambda})=\varepsilon_\lambda p_{\lambda}$, we have
\begin{eqnarray}\label{w(ed)}
\omega\big(h^{\dd}_{\lambda}\big)=
\Sum_{\mu}\varepsilon_{\mu}z^{-1}_{\mu}D^{\dd}_{\mu}R_{\mu\lambda}p_{\mu}
=\Sum_{\nu}\Sum_{\mu}\varepsilon_{\mu}z^{-1}_{\mu}D^{\dd}_{\mu}R_{\mu\lambda}R_{\mu\nu}m_{\nu}.
\end{eqnarray}
By Proposition \ref{m-expansion-y},
\[
\Prod_{i,j}\frac{1-(x_{i}t_j)^{d+1}}{1-x_{i}t_j}=H^\dd(t_1)H^\dd(t_2)\cdots=\Sum_{\lambda}h^\dd_\lambda m_\lambda(t_1,t_2,\ldots).
\]
By Theorem \ref{main-2}, $\omega(H^\dd(t_1)H^\dd(t_2)\cdots)=(H^\dd(-t_1)H^\dd(-t_2)\cdots)^{-1}$. Thus if $d$ is odd, we have
\begin{eqnarray*}
\Sum_{\lambda}\omega\big(h^{\dd}_{\lambda}(x)\big)m_{\lambda}(t_1,t_2,\ldots)
&=&\Prod_{i,j}\frac{1+x_{i}t_{j}}{1-(-x_{i})^{d+1}t_{j}^{d+1}}\\
&=&\Prod_{i,j}\Sum_{k\geq 0}\left((x_{i}t_{j})^{k(d+1)}+(x_{i}t_{j})^{k(d+1)+1}\right)\\
&=&\Sum_{\lambda}\Sum_{\mu}N^{\dd}_{\lambda\mu}m_{\mu}(x)m_{\lambda}(t_1,t_2,\ldots)
\end{eqnarray*}
It follows that
\[
\omega(h^{\dd}_{\lambda})=\Sum_{\mu}N^{\dd}_{\lambda\mu}m_{\mu}.
\]
Comparing with (\ref{w(ed)}), the result follows.
\end{proof}
To end this section, we give an identity on $m_{\lambda}(\xi,\ldots,\xi^{d})$ whose computation is messy in general. If the partitions $\lambda=(\lambda_1,\lambda_2,\ldots)$ and $\mu^i\vdash\lambda_i$ for each $i$, after coordinates reordering, the partition $\mu=(\mu^1,\mu^2,\ldots)$ is called a partition of $\lambda$ and denoted $\mu\vdash \lambda$.

\begin{proposition}\label{identity}
Let $\xi=e^{2\pi i/d+1}$ and $\eta=e^{2\pi i/d'+1}$ be primitive unit roots, and denote $\vec{\xi}=(\xi,\ldots,\xi^d)$ and $\vec{\eta}=(\eta,\ldots, \eta^{d'})$. For any partition $\mu\vdash n$ with $\ell(\mu)\le dd'$, we have
\[
\Sum_{\mu=\left(\mu^1,\mu^2,\ldots\right)\vdash \lambda}\left(m_{\lambda}(\vec{\xi})
\Prod_{i=1}^{\ell(\lambda)}m_{\mu^{i}}(\vec{\eta})\right)=\Sum_{\mu
=\left(\mu^1,\mu^2,\ldots\right)\vdash \lambda}\left(m_{\lambda}(\vec{\eta})
\Prod_{i=1}^{\ell(\lambda)}m_{\mu^{i}}(\vec{\xi})\right).
\]
\end{proposition}
\begin{proof}
Let $\xi=e^{2\pi i/d+1}$ be the $(d+1)$-th primitive unit root. Since $\Prod_{i,j}(1+x_iy_j)=\Sum_\lambda e_\lambda(x) m_\lambda(y)$, we have
\begin{eqnarray*}
\Sum_{n\geq 0}h^{\dd}_{n}=\Prod_{i\geq 1}\Prod_{j=1}^{d}(1-x_{i}\xi^{j})=\Sum_{\lambda }(-1)^{|\lambda|}m_{\lambda}(\vec{\xi})e_{\lambda}(x),\\
\end{eqnarray*}
which implies $h^{\dd}_{n}=\Sum_{\lambda\vdash n}(-1)^{n}m_{\lambda}(\xi,\ldots,\xi^{d})e_{\lambda}$.
\begin{eqnarray*}
\omega^\d'\omega^\dd(e_{n})&=&\omega^{\d'}(h^{\dd}_{n})=\omega^{\d'}\left(\Sum_{\lambda\vdash n}(-1)^{n}m_{\lambda}(\vec{\xi})e_{\lambda}\right)
=\Sum_{\lambda\vdash n}(-1)^{n}m_{\lambda}(\vec{\xi})h^{\d'}_{\lambda}\\
&=&\Sum_{\lambda\vdash n}(-1)^{n}m_{\lambda}(\vec{\xi})\Prod_{i=1}^{\ell(\lambda)}\Sum_{\mu^i\vdash \lambda_{i}}(-1)^{\lambda_i}m_{\mu^i}(\vec{\eta})e_{\mu^i}\\
&=&\Sum_{\lambda\vdash n}\Sum_{\mu=\left(\mu^1,\mu^2,\ldots \right)\vdash \lambda}\left(m_{\lambda}(\vec{\xi})
\Prod_{i=1}^{\ell(\lambda)}m_{\mu^{i}}(\vec{\eta})\right)e_{\mu}.
\end{eqnarray*}
Similarly, we have
\[
\omega^{\dd}\omega^{\d'}(e_{n})=\Sum_{\lambda\vdash n}\Sum_{\mu=\left(\mu^1,\mu^2,\ldots \right)\vdash \lambda}\left(m_{\lambda}(\vec{\eta})
\Prod_{i=1}^{\ell(\lambda)}m_{\mu^{i}}(\vec{\xi})\right)e_{\mu}.
\]
Since $\omega^\d'\omega^\dd(e_{n})=\omega^{\dd}\omega^{\d'}(e_{n})$ by Proposition \ref{diagram-commute}, the result holds.
\end{proof}
\begin{example}Suppose $\mu=(1^{dd'})$. Note $m_\lambda(x_1,\ldots, x_l)\ne 0$ iff $\ell(\lambda)\le l$. From Proposition \ref{identity}, we have
\[
m_{(d'^{d})}(\vec{\xi})\Prod_{i=1}^{d}m_{(1^{d'})}(\vec{\eta})
=m_{(d^{d'})}(\vec{\eta})\Prod_{i=1}^{d'}m_{(1^{d})}(\vec{\xi})=1.
\]
\end{example}
\section{Acknowledgements}
We would like to thank Suijie Wang for his encouragement and help. None of this work would be possible without his insightful guidance. Thanks also goes to
Yue Zhou for reading the drafts and providing valuable comments.  In addition, we are particularly grateful to the anonymous referee for many kind suggestions.

\end{document}